\documentclass[10pt]{amsart}
\usepackage{amsfonts,amssymb,amsmath, forest}
\usepackage{amssymb, amsmath}
\usepackage{float,epsfig}
\usepackage{xcolor}
\usepackage{ mathrsfs }
\usepackage{graphicx}
\usepackage{hyperref}
\usepackage{mathabx}
\usepackage[left=2.5cm,top=1 cm,right=2.5cm,nofoot]{geometry}

\usepackage{cleveref}
\usepackage{relsize}
\usepackage[mathlines]{lineno}
 
\newtheorem{thm}{Theorem}[section]

\newtheorem{lem}[thm]{Lemma}

\newtheorem{deff}[thm]{Definition}

\newtheorem{prop}[thm]{Proposition}
\newtheorem{cor}[thm]{Corollary}

\newtheorem{rk}[thm]{Remark}
\newtheorem{theorem}{Theorem}[section]

\newtheorem{example}[theorem]{Example}

\newenvironment{defn}{\begin{deff}
		\rm }{\end{deff}}
\newcommand{\mc}{\mathcal}

\renewcommand{\ss}{\subseteq}

\newcommand{\ra}{\rightarrow}
\newcommand{\msc}{\mathscr}
\newcommand{\ol}{\overline}

\newcommand{\mfrak}{\mathfrak}

\def \l2x{L^2(X;\mc H)}

\def \sp{{\textrm{span}}}

\def \Pf{{\mbox{{\bf Pf}}}}
\def \HS{{\mc HS}}

\theoremstyle{definition}

\theoremstyle{remark}

\numberwithin{equation}{section}

\usepackage{scalerel,stackengine}
\stackMath
\newcommand\reallywidehat[1]{%
	\savestack{\tmpbox}{\stretchto{%
			\scaleto{%
				\scalerel*[\widthof{\ensuremath{#1}}]{\kern.1pt\mathchar"0362\kern.1pt}%
				{\rule{0ex}{\textheight}}
			}{\textheight}%
		}{2.4ex}}%
	\stackon[-6.9pt]{#1}{\tmpbox}%
}
\begin{document}

\title[Characterizations of  Extra-invariant spaces ]{Characterizations of  Extra-invariant spaces under  \\ the left translations
	on  a  Lie group
}
\author{Sudipta Sarkar}
\address{Department of Mathematics,
	Indian Institute of Technology Indore,
	Simrol, Khandwa Road,
	Indore-453 552}
\email{phd1701141004@iiti.ac.in, nirajshukla@iiti.ac.in}

\thanks{Research of  S. Sarkar and N. K. Shukla was supported by   research grant from CSIR, New Delhi [09/1022(0037)/2017-EMR-I] and NBHM-DAE [02011/19/2018-NBHM(R.P.)/R\&D II/14723], respectively.}

\author{Niraj K. Shukla}





\begin{abstract} 
	In the context of a connected, simply connected, nilpotent Lie group, whose representations are square-integrable modulo the center, we find characterization results  of  extra-invariant spaces under  the left translations
	associated with the range functions.  Consequently, the theory is valid for the Heisenberg group $\mathbb H^d$, a 2-step nilpotent Lie group. 
\end{abstract}


\maketitle

\section{{\bf Introduction }}\label{intro}
Let $G$ be a connected, simply connected, nilpotent Lie group with Lie  algebra $\mathfrak{g}$ and   $Z$ be  the  center of $G$.  Then $G$ is an \textit{$SI/Z$ group} if almost all of its irreducible representations are square-integrable (SI) modulo the center $Z$.
 An irreducible representation $\pi$ of $G$ is called \textit{square integrable modulo the center} if  it satisfies the condition 
 \begin{align*}
 \int_{G/Z}|\langle \pi({g}) u,v\rangle|^2d \dot{g}<\infty, \ \mbox{for all}\ u, v.
 \end{align*} 
We briefly start by describing the left translation generated  systems in $L^2 (G)$ as follows: 
  \begin{defn} \label{invariant space}
  	Let  $\Lambda$ be a   uniform lattice in  the center  $Z$  of $G$ and  $\Gamma$ be a discrete set    lying outside the center $Z$. 
 	A closed subspace $W$ of $L^2(G)$ is said to be \textit {$\Gamma\Lambda$-invariant} if 
 	$$
 	L_{\gamma\lambda}f\in W, \  \mbox{for all} \ \gamma\in \Gamma, \lambda\in \Lambda  \ \mbox{and} \ f \in W, 
 	$$ 
 	where for each $y\in G$,  $L_{y} f (x)= f (y^{-1}x),$ for $x \in G$ and $f \in L^2 (G)$.
 \end{defn}
 For a sequence  of functions  $\msc A =\{\varphi_k : k\in  I\}$ in $L^2( G)$,    we   define \textit{$\Gamma\Lambda$-invariant   space  $\mc S^{\Gamma\Lambda}(\msc A)$ generated by $\msc A$} with the action of $\Gamma\Lambda$ as  follows: 
 \begin{align}\label{TIsystem}
 	\mc S^{\Gamma\Lambda}(\msc A) = \ol{\mbox{span}} \ \mc E^{\Gamma\Lambda}(\msc A), \ \mbox{where} \ \mc E^{\Gamma\Lambda} (\msc A) =\{L_{\gamma\lambda}\varphi :    \gamma \in  \Gamma,\lambda\in \Lambda,\varphi \in \msc A\}.
 \end{align}
 If  $\msc A =\{\varphi\}$, we denote $\mc S^{\Gamma\Lambda}(\msc A)$ as  $\mc S^{\Gamma\Lambda}(\varphi)$.  In general,  $\mc E^{\Gamma\Lambda} (\msc A)$  and   $\mc S^{\Gamma\Lambda}(\msc A)$  are  known as  \textit{translation generated (TG) system} and \textit{translation invariant  (TI) space}, respectively.  Currey et al. in \cite{currey2014characterization}  provided a characterization of  all $\Gamma\Lambda$-invariant spaces using range function, followed by the work of Bownik  in  \cite{bownik2000structure}.

 Next we define an invariance set in the center $Z$ of the $SI/Z$ nilpotent Lie group $G$. 
 
 \begin{defn} 
 	For a  given  $\Gamma\Lambda$-invariant subspace $W$,  an \textit{invariance  set}  $\Theta$   is   defined by 
 	\begin{align}\label{invariance set}
 	\Theta=\{\lambda\in Z: L_{\gamma\lambda}f\in W,  \  \mbox{for all} \  \gamma \in  \Gamma \ \mbox{and} \ f\in W\}.
 \end{align}
 \end{defn}
The set $\Theta$  is a closed subgroup of $Z$ containing  $\Lambda$.  For this let us consider a net $(\lambda_\alpha)$ in $\Theta$ such that  $\lim_{\alpha} \lambda_\alpha =\lambda$, say. Then   we have $\lim_{\alpha}\|L_{\gamma\lambda_\alpha }f-L_{\gamma\lambda}f\|=0,$ for $f \in W$ and $\gamma \in \Gamma$, and hence,    $\lambda \in \Theta$ since $W$ is a  closed subspace. Therefore,  $\Theta$  is a closed set.  Since   $\Theta$ is a semigroup of $Z$ and     the image of quotient map  from $Z$ to $Z/\Lambda$ on   $\Theta$ is closed in  $Z/\Lambda$ and hence compact, therefore, the group property of $\Theta$ follows  from  the fact that  a  compact semigroup of $Z/\Lambda$ is   a group.

This paper aims to characterize all $\Gamma\Lambda$-invariant spaces  $W$ to become $\Gamma\Theta$-invariant, where $\Theta$ is a closed subgroup of $Z$ and  $\Lambda\subset \Theta$. This is known as  the \textit{extra-invariance}  of  a translation invariant space $W$. The current study of extra-invariance   touches on the non-abelian setup for a nilpotent Lie group, which is considered to be a high degree of non-abelian structure. Consequently, the theory is valid for the $d$-dimensional Heisenberg group $\mathbb H^d$, a 2-step nilpotent Lie group. The history of extra-invariance dates back to the work of Aldroubi et al. in \cite{aldroubi2010invarianceR} for the one-dimensional Euclidean case, and Anastasio et al. in \cite{extraRn, extraLCA} for higher dimensions and locally compact abelian groups.

We find necessary and sufficient conditions under which a $\Gamma\Lambda$-invariant space becomes $\Gamma\Theta$-invariant  in the context of nilpotent Lie group $G$ whose representations are $SI/Z$ type.  The characterization results below are based on the Plancherel transform.      Unlike the Euclidean and LCA group setup,  the Plancherel transform of a function is operator-valued so that the technique used in the Euclidean and LCA groups is restrained. We now state the main results of the paper.

 \begin{thm}\label{extra-invariance result}  	Let $\Lambda$ be  a   uniform lattice in the center  $Z$ of $G$ and  $\Gamma$ be a discrete  set    lying outside the center $Z$ containing the identity element $e$ such that   $\mc J$and $\Sigma$ are the Borel sections of $\Lambda^\perp/\Theta^\perp$ and $\widehat Z/ \Lambda^\perp$, respectively, where $\Theta$ is a  closed subgroup of $Z$ containing $\Lambda$.  If  $W$ is  a $\Gamma\Lambda$-invariant subspace of $L^2(G)$, then it is     $\Gamma \Theta$-invariant   if and only if for each $j\in \mc J$, 
  $W$ contains  $V_j^\Theta$,  	  where 
  \begin{align*} 
		V_j^\Theta=\{f\in L^2(G): \widehat f= \chi_{\mc H_{j}^\Theta}\widehat g, \mbox{with} \  g \in W\}, \ \mbox{and} \ \mc H_j^{\Theta}:=\Sigma+j+\Theta^\perp.  
	\end{align*}
 In this case, the space  $W$ can be decomposed as the orthogonal direct sum of $V_j^\Theta$'s,    i.e., 
	$ W= \displaystyle\bigoplus_{j\in \mc J}V_j^\Theta.$
\end{thm}
As a consequence, we find the below characterization result for $\mc S^{\Gamma\Lambda}(\msc A)$ to become  $\Gamma \Theta$-invariant using the associated range function.

	\begin{thm}\label{TI_extra-invariance result}	In addition to the hypothesis  of Theorem \ref{extra-invariance result}, let $\msc A$ be a sequence of functions in $L^2(G)$. Then 
		 $\mc S^{\Gamma\Lambda}(\msc A)$ is a $\Gamma \Theta$-invariant if and only if the Plancherel transform followed by periodization  $\mc T^{\Lambda}$ satisfies 
		 $$\mc T^{\Lambda}(L_{\gamma}\varphi^j)(\sigma)\in J(\sigma),  \ a.e.   \ \sigma\in \Sigma,   \ \mbox{for all} \  j \in \mc J   \ \mbox{and} \   \gamma \in \Gamma,
		 $$  where the associated range function     $J(\sigma)=\ol{\sp} \{\mc T^{\Lambda}(L_{\gamma}\varphi)(\sigma):\varphi\in \msc A, \gamma\in\Gamma\},$ and    $\widehat \varphi^j=\widehat \varphi \chi_{\mc H^{\Theta}_j}$. 
\end{thm}
	We further characterize this extra invariance using the dimension function.
 Given any $\Gamma\Lambda$-invariant subspace $W$ of $L^2(G)$,  we define the \textit{dimension function} as $$\dim_W:\Sigma\ra \mathbb N_{0}\bigcup \{\infty\} \ \mbox{by} \  \dim_W(\sigma):=\dim(J_W(\sigma)),  \ \mbox{for a.e.} \ \sigma \in\Sigma, $$ 
where  $J_W (\sigma)$ is the range function associated to  $W$.

\begin{thm}\label{dimension} Under the standing hypothesis of Theorem \ref{extra-invariance result}, the $\Gamma\Lambda$-invariant space $W$ is   $\Gamma\Theta$-invariant if and only if  	 the dimension function satisfies the following relation
	$$\dim_W(\sigma)=\sum_{j\in\mc J} \dim_{V_j^\Theta}(\sigma), \  a.e.  \ \sigma\in \Sigma.
	$$ 
\end{thm}
 
 The following result can be established easily for the $d$-dimensional Heisenberg group $\mathbb H^d$, a 2-step nilpotent Lie group, using Theorems  \ref{extra-invariance result}, \ref{TI_extra-invariance result} and \ref{dimension}. 
 The \textit{$d$-dimensional Heisenberg group},  denoted by $\mathbb  H^d$, is an example of $SI/Z$ group. The group $\mathbb H^d$ can be   identified with $\mathbb R^d\times \mathbb R^d\times \mathbb R$ 
 under the group operation 
 $(x,y,w).(x',y',w')=(x+x',y+y',w+w'+x\cdot y)$, $x, x', y, y' \in \mathbb R^d$, $w, w' \in \mathbb R$,  where $\cdot$ stands for $\mathbb R^d$ scalar product.  For this  case, the uniform lattice is $\Lambda=\mathbb Z$,  $\Gamma$ is a discrete set of the form $A\mathbb  Z^d \times B \mathbb Z^d$,  and the extra invariance set  $\Theta$ is of the form $\frac{1}{N}\mathbb Z$,  where  $A, B \in GL(d, \mathbb R)$ with $A B^t \in \mathbb Z$,  and $N \in \mathbb N$.  

\begin{thm} 
	Let $A, B\in GL(d,\mathbb R)$ such that  $AB^t\in \mathbb Z$ and let $N \in \mathbb N$. 
	If  $W$ is  an $A\mathbb Z^d\times B\mathbb Z^d \times \mathbb Z$-invariant subspace of $L^2(\mathbb H^d)$, then it is     $A\mathbb Z^d\times B\mathbb Z^d \times \frac{1}{N}\mathbb Z$-invariant  if and only if for each $n\in \mathbb Z_N:=\{0, 1, 2, \cdots, N-1\}$, 
	$W$ contains  $V_n^{\frac{1}{N}\mathbb Z}$,  	  where 
	\begin{align} 
	V_n^{\frac{1}{N}\mathbb Z}=\{f\in L^2(G): \widehat f= \chi_{\mc H_n^{\frac{1}{N}\mathbb Z}}\widehat g, \mbox{with} \  g \in W\}, \ \mbox{and} \ \mc H_n^{\frac{1}{N}\mathbb Z} =[n, n+1)+N\mathbb Z.  
	\end{align}
	In this case, the space   
	$ W= \displaystyle\bigoplus_{n\in \mathbb Z_N}V_n^{\frac{1}{N}\mathbb Z},$ and   $
	\dim_W(\xi)=\sum_{n\in\mathbb Z_N} \dim_{V_n^{\frac{1}{N}\mathbb Z}}(\xi), \  a.e.  \ \xi \in [0, 1)$. 
	\end{thm}

As an application of the above results, the following consequence provides an estimate to measure the support of the Plancherel transform of a  generator of $\mc S^{\Gamma\Lambda}(\msc A)$.

\begin{thm}\label{measure}  In addition to the hypothesis of  Theorem \ref{extra-invariance result}, let $\msc A=\{\varphi_i\}_{i=1}^n \subset L^2(G)$ and   $\Gamma$ be a finite set having cardinality $k$, i.e., $|\Gamma|=k$.  If  $\mc S^{\Gamma\Lambda}(\msc A)$ be   $\Gamma\Theta$-invariant,  the following inequality holds:
	\begin{align}\label{estimate}
	 \mu(\{\delta\in \mc D: \widehat \varphi_i(\delta)\neq 0\})\leq \sum_{m=0}^{nk} m \, \mu(\Sigma_{m})   \leq nk, \ \mbox{for all}\ i \in \{1, 2, \cdots, n\}, 
	\end{align}
	 where $\mc D$ is the  Borel section of $\widehat Z/\Theta^\perp$,  $\Sigma_{m}=\{\sigma\in\Sigma: \dim _W(\sigma)=m \}$ and $0$ is in the sense of the zero operator.
\end{thm}

  The paper is organized as follows: In Section \ref{S:Pre}  we discuss the irreducible representations of a nilpotent Lie group and  Plancherel transform for the $SI/Z$ group.  Employing the Plancherel transform  followed by periodization, we  establish the proof of our main results Theorem \ref{extra-invariance result}, \ref{TI_extra-invariance result}, \ref{dimension} and \ref{measure} in Section \ref{Proof} by involving the range function associated to a $\Gamma \Lambda$-invarinat space. 
  
   \section{Irreducible   representations  of  a nilpotent Lie group}\label{S:Pre} Let $G$ be a connected, simply connected, nilpotent Lie group with Lie  algebra $\mathfrak{g}$. We identify 
$G$ with $\mfrak g\cong \mathbb R^n$ due to the analytic diffeomorphism of the exponential map $\exp:\mathfrak g\rightarrow G$, where $n=\mbox{dim}~{\mfrak g}$. To choose a basis for the Lie  algebra $\mathfrak{g}$, we consider the Jordon-H\"older series $(0)\ss \mfrak g_1\ss \mfrak g_2\ss\dots \ss \mfrak g_n=\mathfrak g$ of ideals of $\mfrak g$   such that dim~$\mfrak g_j=j$ for $j=0,1,\dots, n$ satisfying  ad$(X)\mfrak g_j\ss \mfrak g_{j-1}$  for $j=1,\dots, n$, for all $X\in \mfrak g$, where for $X, Y\in \mfrak g$, ad$(X)(Y)=[X, Y]$, the Lie bracket of $X$ and $Y$.  Now   we pick $X_j\in \mathfrak g_j\backslash \mathfrak g_{j-1}$  for each $j=1,2,\dots, n$, such that the collection $\{X_1, X_2, \dots, X_n\}$ is a Jordan-H\"older basis.  Then  the map $\mathbb R^n\longrightarrow \mathfrak g\longrightarrow G$ defined by $(x_1,x_2,\dots,x_n)\mapsto \sum_{j=1}^n x_jX_j\mapsto \exp(\sum_{j=1}^n x_jX_j)$ is a diffeomorphism, and hence the Lebesgue measure on $\mathbb R^n$ can be realized as a Haar measure on $G$. 

Note that the center $\mathfrak z$ of the Lie algebra $\mathfrak{g}$ is non-trivial, and it maps to the center $Z:=\exp{\mathfrak z}$ of $G$.
 The Lie group $G$ acts on $\mathfrak g$  and $\mathfrak g^*$ by the \textit{adjoint action} $\exp(Ad(x)X):=x\exp(X)x^{-1}$ and   \textit{co-adjoint action} $(Ad^*(x)\ell)(X) =\ell(Ad(x^{-1})X)$, respectively,  for  $x\in G, X\in \mathfrak g,$ and $ \ell\in \mathfrak g^*$.     The  $\mathfrak g^*$ denotes  the  vector space of all real-valued linear functionals on $\mathfrak g$.  For  $\ell\in \mathfrak g^*$ the   \textit{stabilizer} $R_{\ell} =\{x\in G:(Ad^*x)\ell=\ell\}$ is a Lie group   with the associated Lie algebra $r_{\ell}:=\{X\in \mathfrak{g}: \ell[Y, X]=0, \mbox{for all} \ Y \in \mathfrak g\}$. 


Our aim is to discuss   Kirilov Theory \cite{corwin2004representations}   to define the Plancherel transform for $SI/Z$ group. Given any $\ell\in \mathfrak{g^*}$ there exists a subalgebra~ $\mfrak{h}_{\ell}$ (known as \textit{polarizing} or \textit{maximal subordinate  subalgebra}) of $\mathfrak{g}$ which is maximal with respect to the property
$\ell[\mathfrak{h}_{\ell},\mfrak{h}_{\ell}]=0$. Then the map  $\mathcal{X}_{\ell}:\exp(\mfrak{h}_{\ell})\rightarrow \mathbb T$ defined  by
$\mathcal X_{\ell}(\exp X)=e^{2\pi i \ell(X)}$,  $X\in \mfrak{h}_{\ell}$ is a character on $\exp(\mfrak{h}_{\ell})$, and hence the representations induced from $\mathcal X_{l}$, $\pi_{\ell} : =\mbox{ind}_{\exp{\mathfrak h_{\ell}}}^G$$\mathcal{X}_{\ell}$, have the following properties:
\begin{itemize}
	\item[(i)] $\pi_{\ell}$ is an irreducible unitary representation of $G$.
	\item[(ii)]  Suppose $\mfrak{h'}_{\ell}$ is another subalgebra which is maximal with respect to the property $\ell[{\mathfrak {h}}'_{\ell}, {\mfrak{h}}'_{\ell}]=0$, then $\mbox{ind}_{\exp{\mfrak{h}_{\ell}}}^G\mathcal X_{\ell}\cong\mbox{ind}_{\exp(\mfrak h'_{\ell})}^G\mathcal X_{\ell'}$. 
	\item[(iii)] $\pi_{\ell_1}\cong \pi_{\ell_2}$ if and only if $\ell_1$ and $\ell_2$ lie in the same co-adjoint orbit.
	\item[(iv)] Suppose $\pi$ is a irreducible unitary representation of $G$, then there exists $\ell\in \mfrak{g^*}$ such that $\pi\cong \pi_{\ell}$.
	\end{itemize}
	Therefore there exists  a   bijection 
$\lambda^*:\mfrak g^*/Ad^*(G)\rightarrow \widehat{G}$ which is also  a Borel isomorphism, where $\widehat{G}$ is the collection of all irreducible unitary representations of $G$.

   For an irreducible representation $\pi\in \widehat G$, let $O_{\pi}$ denote as coadjoint orbit corresponding to the equivalence class of $\pi$. Then the orbital characterization for the SI/Z representation is: 
$\pi$ is square integrable modulo the center if and only if  for $\ell\in O_\pi$, $r_{\ell}=\mathfrak z$ and $O_\pi=\ell+\mathfrak z^{\perp}$. If SI/Z $\neq \phi$,  then SI/Z=$\widehat{G}_{max}$, where the Borel subset  $\widehat {G}_{max}\ss \widehat G$ corresponds  to coadjoint orbits of maximal dimension which is co-null for Plancherel measure class. Hence when $G$ is an SI/Z group,  $\widehat G_{max}$ is parameterized by a subset of $\mathfrak z^*$. If  $\pi \in \widehat G_{\max}$, then  dim $O_\pi=n-\mbox{dim} \  \mathfrak z$, since $O_\pi$ is symplectic manifold, it is of even dimension, say,  $\dim O_\pi=2d$.  By Schurs' Lemma the restriction of $\pi$ to $Z$ is a  character and hence  it is a unique element $\sigma=\sigma_\pi \in \mathfrak z^*$ (say) and $\pi(z)=e^{2\pi i\langle \sigma, log z\rangle}I$, where $I $ is the identity operator. It shows that $O_\pi=\{l\in \mathfrak g^*:l|_{\mathfrak z}=\sigma\}$ and $\pi\mapsto \sigma_\pi$ is injective.

 Let $G$ be  an SI/Z group    and $\mc W=\{\sigma \in \mathfrak{z^*}:{\bf Pf}(\sigma)\neq 0\}$
 be a cross section for the coadjoint orbits of maximal dimension, where  the \textit{Pfaffian determinant}   ${\bf Pf}: \mathfrak z^*\longrightarrow \mathbb R$ is given by  $\ell \mapsto \sqrt{\big| \det(\ell[X_i, X_j] )_{i,j=r\dots n}\big|}$.  Then  for a fixed $\sigma \in \mc W$,  $p(\sigma)=\sum_{j=1}^d \mathfrak g_j(\sigma\big|_{\mathfrak g_j})$ is a  maximal subordinate subalgebra for $\sigma$ and the corresponding induced representation $\pi_\sigma$ is realized naturally in $L^2 (\mathbb R^d)$, where $n=r+2d$ for some $d$. For each $\varphi \in L^1(G)\cap L^2(G)$, the Fourier transform of $\varphi$  given  by  
$$\widehat{\varphi}(\sigma)=\int_{G} \varphi(x)\pi_{\sigma}(x)dx, \quad \sigma \in \mc W$$
defines a trace-class operator on $L^2(\mathbb R^d)$, with the inner product $\langle A, B\rangle_{ \HS}=tr(B^*A)$. This space is denoted by $\HS(L^2(\mathbb R^d))$.  When $d\sigma$ is suitable normalized,
$$ \|\varphi\|^2 =\int_{\mc W}\|\widehat{\varphi}(\sigma)\|^2_{\HS(L^2(\mathbb R^d))}|{\bf Pf}(\sigma)| d(\sigma).$$
 The Fourier transform can be extended unitarily as $\mc F$ - the \textit{Plancherel transform},
 \begin{align*}
 	\mc F:L^2(G) & \rightarrow L^2(\mathfrak z^*,\HS(L^2(\mathbb R^d)), |{\bf Pf}(\sigma)|d\sigma)) ,   \quad  \mc Ff=\widehat f.
 	 \end{align*}
  Note that the Plancherel transform  $\mc F$ satisfies the relation 
  $$\mc F (L_{\gamma} f )(\sigma)=\pi_{\sigma}(\gamma) \mc F f (\sigma), \  \mbox{for} \ \gamma \in G,  \ a.e. \  \sigma \in \mathfrak z^*,  \ \mbox{and} \  f \in L^2(G),$$
    where   the left translation operator $L_\gamma$ on $L^2 (G)$  is given by $L_{\gamma} f (x)= f (\gamma^{-1}x)$.  
 
In case of $\mathbb H^d$, $\mc W= \mathbb R\backslash \{0\}$ and  $|\Pf(\sigma)|=|\sigma|^{d}$. Then    the Plancherel transform for $\varphi \in L^1(\mathbb H^d)\cap L^2(\mathbb H^d)$ is defined by $
\mc F\varphi(\lambda)=\int_{\mathbb H^d}\varphi(x)\pi_{\lambda}(x)dx$ for $\lambda \in \mc W= \mathbb R\backslash \{0\}$,  where the Schr\"{o}dinger representations $\pi_{\lambda}$ on $L^2(\mathbb R^d)$ is given by 
$\pi_{\lambda}(x,y,w)f(v)=\pi_{\lambda}(x,y,w)f(v)=e^{2\pi i \lambda w}e^{-2\pi i \lambda y. v}f(v-x), \ (x,y,w)\in \mathbb H^d, \ f \in L^2(\mathbb R^d). $


 

\section{Proof of  the Main  Results}\label{Proof}

 \textit{Through out the paper}  let us assume that  $G$ be  an SI/Z  nilpotent Lie group with center $Z$.  From the Section \ref{S:Pre},  we consider  the center  $Z$ identified with $\mathbb R^{r}$  ($r<n$)  for  a chosen  (ordered) basis $\{X_1, X_2,\dots, X_n\}$ of the corresponding  Lie algebra $\mathfrak g$, as follows: $$Z=\exp \mathbb RX_{1} \exp \mathbb RX_{2} 
 \cdots \exp \mathbb RX_r.$$ Also we write a set   $\mathcal X$  identified with $\mathbb R^{2d}$ ($n=r+2d$) as follows:   
 $$\mathcal X = \exp \mathbb RX_{r+1} \exp \mathbb RX_{r+2}\cdots \exp \mathbb RX_n.$$   
 The elements  $y=(y_1, y_2, \dots, y_r) \in \mathbb R^r$ and $x=(x_{r+1}, \dots, x_n) \in \mathbb R^{2d}$  are   identified by   
 $$
 y=\exp y_1  X_1 \exp y_2 X_2 \cdots  \exp y_r X_r,   \ \mbox{and} \ x=\exp  x_{r+1}X_{r+1} \exp x_{r+2} X_{r+2}\cdots \exp x_{n} X_n.
 $$
 It can be noted from   the homeomorphism between $\mathbb R^r \times \mathbb R^{2d}$ and  $G$ given by 
 {\small $$
 	(y_1, y_2, \dots, y_r, x_{r+1}, \dots, x_n)\mapsto \exp y_1  X_1 \exp y_2  
 	X_2 \cdots  \exp y_r X_r \exp  x_{r+1}X_{r+1} \exp x_{r+2} X_{r+2}\cdots \exp x_{n} X_n. 
 	$$
 }
Further assume that     $\Lambda$ is  a   uniform lattice in   $Z$, means, it is a discrete closed subgroup of $Z$  such that $Z/\Lambda$ is compact.  Then we have $\widehat{Z/\Lambda}\cong \Lambda^\perp$ and $\widehat Z/\Lambda^\perp\cong \widehat \Lambda$  since the center  $Z$  becomes a  locally compact abelian group. The dual group of  Z, denote by $\widehat{Z}$, is also identical with $\mathbb R^r$. The group $\widehat{Z}$ consists of continuous homomorphisms from $Z$ to $\mathbb T$, and the \textit{annihilator} $\Lambda^\perp$  is defined by  
$$	
\Lambda^\perp =\{\lambda^* \in   \widehat{Z}:  \lambda^*(\lambda)=1, \ \forall \ \lambda\in \Lambda\}. 
$$
 The set  $\widehat{Z}$ can be tiled by $\Sigma$ with  the tiling   partner $\Lambda^\perp$, where  $\Sigma$ is  \textit{measurable section} of  $\widehat Z/\Lambda^\perp$ having finite measure.  The set $\Sigma$ is  a tiling set  of $\widehat {Z}$, means, the collection $\{\Sigma+\lambda^*: \lambda^* \in \Lambda^\perp\}$ is a measurable partition of  $\widehat Z$.   For a measure space  $(X,\mu)$,  a countable set $\{\Omega_j\}_j$ of subsets of $X$ is    \textit{tiling} of $X$ if 
 $\mu(X\backslash\bigcup_{j}\Omega_j)=0$, and $\mu(\Omega_i\cap \Omega_j)=0$, when $i\neq j.$ A set $T$    is      \textit{tiling partner}  of $\Omega$  for $X$  if  there is a set $\Omega$ in $X$ such that   the collection $\{\Omega+x: x\in T\}$ is a tiling of $X$.

We also fix the set  $\Gamma$   in $ \mathcal X$   lying  outside of  the center $Z$.  Recall  the $\Gamma\Lambda$-invariant  subspace $W$ from Definition  \ref{invariant space}, $\Gamma\Lambda$-invariant space $\mc S^{\Gamma\Lambda}(\msc A)$ from (\ref{TIsystem}) generated by $\msc A$ and invariance set $\Theta$ from (\ref{invariance set}). Firstly we concentrate for the properties of $\Theta$.    We observe the tilling  property of  $\Sigma$ with respect to $\Theta^\perp$ in the following result.

\begin{prop}
	For  any section $\mc J$ of  $\Lambda^\perp/\Theta^\perp$, the set $\Theta^\perp+\mc J$ is a tiling partner of  $\Sigma$ for $\widehat Z$.  That means,   the collection  $\{\mc H_j^{\Theta}\}_{ j\in \mc J}$ is tiling of   $\widehat Z$, where 
	\begin{align}\label{Hj}
	H_j^{\Theta}:=\Sigma+j+\Theta^\perp. 
	\end{align}
\end{prop}

\begin{proof}  The set $\Sigma$ is  a tiling set  of $\widehat {Z}$, means, the collection $\{\Sigma+\lambda^*: \lambda^* \in \Lambda^\perp\}$ is a measurable partition of  $\widehat Z$. Since $\widehat{Z/\Lambda}\cong \Lambda^\perp$ and $Z/\Lambda$ is compact, therefore $\Lambda^\perp$ is discrete and countable,  and hence    $\Theta^\perp$ is also discrete and countable  follows from $\Theta^\perp
	\subset \Lambda^\perp$. Hence the collection $\{\Theta^\perp +j: j \in \mc J\}$ is a  tiling of $\Lambda^\perp$ by considering a Borel section  $\mc J$ of  $\Lambda^\perp/\Theta^\perp$.  Thus the result follows  by  employing the fact  $\Lambda^\perp$ is tiling partner of $\Sigma$ for $\widehat Z$ and $\mc J$ is a  tiling  partner of  $\Theta^\perp$  for $\Lambda^\perp$. 	
\end{proof}
\begin{example}
 For the Heisenberg group $\mathbb  H^d$, the uniform lattice $\Lambda$ in the center $ \mathbb R$, is the set of all integers $\mathbb Z$. Since the only proper closed additive subgroups of $\mathbb R$ containing $\mathbb Z$ are $\frac{1}{N}\mathbb Z$, for some natural number $N$, we consider the extra-invariance set $\Theta=\frac{1}{N} \mathbb Z$. Then the annihilators of $\Lambda$ and $\Theta$ are $\Lambda^\perp=\mathbb Z$ and  $\Theta^\perp=N\mathbb Z$, respectively. Note that the set $\mathbb R$ can be tiled by a Borel section $\Sigma=[0,1)$ with the tiling partner $\Lambda^\perp=\mathbb Z$. By assuming the Borel section $\mathbb Z_N:=\{0,1,\dots, N-1\}$ of $\Lambda^\perp/\Theta^\perp=\mathbb Z/N\mathbb Z$,  the set $n \mathbb Z+\mathbb  Z_N$ is a tiling partner of $[0, 1)$ for $\mathbb R$, that means, the collection $\{\mc H_n^{\frac{1}{N} \mathbb Z}\}_{n=0}^{N-1}$  is tiling of $\mathbb R$, where 
\begin{align*}
\mc H_n^{\frac{1}{N} \mathbb Z} =[0,1)+n+N\mathbb Z= \bigcup_{k\in \mathbb Z}[n, n+1)+Nk.
\end{align*}
For the case of Heisenberg group $\mathbb  H^d$, we consider the set $\Gamma=A\mathbb Z^d\times B\mathbb Z^d$ from outside of the center of $\mathbb H^d$,  where $A, B\in GL(d,\mathbb R)$ such that  $AB^t\in \mathbb Z$.
\end{example}

In the present section our first goal is to prove Theorem \ref{extra-invariance result} which   characterizes  invariant subspaces of  $L^2(G)$ with the action of   $\Theta$ of the center $Z$ containing  the uniform lattice  $\Lambda$. The following  lemma plays a crucial role to establish the Theorem \ref{extra-invariance result}.

 \begin{lem}\label{lemmaWj}
 	Let $W$ be a $\Gamma\Lambda$-invariant subspace of $L^2(G)$ and let $\mc J$ be a Borel section of  $\Lambda^\perp/\Theta^\perp$. For  each $j\in \mc J$,  consider the space  $V_j^\Theta$ given by 
 	\begin{align}\label{Wj}
 		V_j^\Theta=\{f\in L^2(G): \widehat f= \chi_{\mc H_j^{\Theta}}\widehat g, \ \mbox{for some} \  g \in W\} ,
 	\end{align}
 where $\mc H_j^{\Theta}$ is defined in (\ref{Hj}).  If  $V_j^\Theta \subset W$, it 	is a $\Gamma\Theta$-invariant (and hence,  $\Gamma\Lambda$-invariant)  subspace of $L^2(G)$.
 \end{lem}
For the proof of Lemma \ref{lemmaWj}, we discuss the Plancherel transform followed by a periodization  $\mc T^{\Lambda}$  and the range function $J$ associated to a  $\Gamma\Lambda$-translation generated space $\mc S^{\Gamma\Lambda}(\msc A)$. First we state the Plancherel transform followed by a  periodization $\mc T^\Lambda$, which is an operator-valued linear isometry.   It is well known but for the sake of completion, we provide its proof  with the approach of  the composition of  unitary maps.  The map $\mc T^\Lambda$   intertwines left translation with a representation $\tilde{\pi}$.  For  $g \in G$ and  $h\in  L^2( \Sigma,\ell^2( \Lambda^\perp,\mc HS(L^2(\mathbb R^d)))$, the representation $\tilde{\pi}$ is given by 
$$
\tilde{\pi}(g) h(\sigma)=\tilde{\pi}_{\sigma}(g)h(\sigma), \  a.e. \  \sigma \in \Sigma.
$$
The associated representation  $\tilde \pi_{\sigma}(g)$ on $\ell^2( \Lambda^\perp,\mc HS(L^2(\mathbb R^d))$ is given by 
$$\tilde{\pi}_{\sigma}(g) z(\lambda^*)=\pi_{\sigma+{\lambda^*}}(g) \degree z({\lambda^*}), \quad  {\lambda^*} \in \Lambda^\perp, 
$$ 
where the sequence $(z({\lambda}^\perp)) $ lies in $\ell^2( \Lambda^\perp,\mc HS(L^2(\mathbb R^d))$, $\degree$  denotes the composition of operators in $\HS(L^2(\mathbb R^d))$ and $\pi_{\sigma+{\lambda^*}}(g)$ is the Hilbert-Schmidt operator defined on $L^2(\mathbb R^d)$.  
 
\begin{prop}\label{fiberization} (i) There is a unitary map $\mc T^{\Lambda}:L^2(G)\rightarrow L^2( \Sigma, \ell^2(  \Lambda^\perp,\mc HS(L^2(\mathbb R^d)))$ given  by 
	\begin{equation*}
		\mc T^{\Lambda}f(\sigma)(\lambda^*)=\mc Ff(\sigma+\lambda^*)|\Pf(\sigma+\lambda^*)|^{1/2}, \quad  f\in L^2(G),  \lambda^* \in \Lambda^\perp \  \mbox{and a.e.} \ \sigma \in \Sigma. 
	\end{equation*}

	(ii) The map $\mc T^{\Lambda}$ satisfies the intertwining property of left translation    with the representation $\tilde \pi_{}$:
	\begin{equation}\label{eq:intertwining}
	\mc T^{\Lambda}(L_{\gamma\lambda}f)(\sigma)=e^{2\pi i\langle \sigma, \lambda\rangle}\tilde{\pi}_{\sigma}(\gamma)\mc T^{\Lambda}f(\sigma), \quad  f \in L^2(G), \  \gamma  \in \Gamma, \lambda \in  \Lambda \ \mbox{and} \ \mbox{a.e.}\  \sigma\in \Sigma. 
	\end{equation} 
\end{prop}
\begin{proof}  (i) Since the set $\mc W$ is Zariski open in $\mathfrak z^*$ and $\Pf(\sigma)$ is non vanishing on $\mc W$, the map 
	$$\mathscr U_2: L^2(\mc W, \HS(L^2(\mathbb R^d)), |\Pf(\sigma)|d\sigma) \ra L^2(\mathfrak z^*, \HS(L^2(\mathbb R^d))),  \quad  h\mapsto h |\Pf(\sigma)|^{1/2}$$
	 is unitary. 
	Further note that   $\Lambda^\perp$ is tiling partner of $\Sigma$ for $\widehat Z \cong \mathfrak z^*$, we can define a periodization map 
	$${\mathscr U_3} : L^2(\mathfrak{z^*},\HS(L^2(\mathbb R^d)))  \ra L^2 \left( \Sigma, \ell^2(\Lambda^\perp, \HS(L^2(\mathbb R^d)))\right),  \quad h \mapsto (h(\cdot +\lambda^*))_{\lambda^*\in \Lambda^\perp}.
	$$
	It  is  also unitary  by identifying  the linear  dual $\mathfrak{z^*}$ of $\mathfrak{z}$   with $\mathbb R^r$. Therefore we get  a sequence of unitary maps as follows:
	{\small
		\begin{align*}	
			L^2(G) \xrightarrow[{}]{\mathscr U_1}L^2\left(\mc W, \mc HS(L^2(\mathbb R^d)), |\Pf(w)|dw\right)\xrightarrow[{}]{\mathscr U_2}L^2(\mathfrak{z^*},\HS(L^2(\mathbb R^d))) \xrightarrow[{}]{\mathscr U_3} L^2 \left( \Sigma,\ell^2(\Lambda^\perp, \HS(L^2(\mathbb R^d)))\right),
		\end{align*}
	}   
	where the unitary map $\mathscr U_1$ is   Plancherel transform $\mc F$.
	
	For  $h \in L^2\left(\mc W, \mc HS(L^2(\mathbb R^d)), |\Pf(\sigma)|d\sigma\right)$, we observe  
	$$
	(\mathscr U_2\mathscr U_1)h(\sigma)=(\mathscr U_1h)(\sigma)|\Pf(\sigma)|^{1/2}=\mc Fh(\sigma)|\Pf(\sigma)|^{1/2}, \ a.e. \ \sigma\in \mc W
	$$
	and then for a.e. $\sigma \in \Sigma$,  $\lambda^* \in \Lambda^\perp$ and $f \in L^2 (G)$, we have 
	$$
	(\mathscr U_3\mathscr U_2\mathscr U_1)f(\sigma)(\lambda^*)=\left(\mathscr U_3(\mathscr U_2\mathscr U_1)f(\sigma)\right)(\lambda^*)=(\mathscr U_2\mathscr U_1)f(\sigma+\lambda^*)=\mc Ff(\sigma+\lambda^*)|\Pf(\sigma+\lambda^*)|^{1/2}.
	$$
	Thus the result follows by choosing   $\mc T^{\Lambda}=\mathscr U_3\mathscr U_2\mathscr U_1$.   
 
\noindent (ii)	For $\lambda^* \in \Lambda^\perp$ and a.e. $\sigma \in \Sigma$, we get    
	\begin{align*}
		\mc T^{\Lambda}(L_{\gamma\lambda}f)(\sigma)(\lambda^*)
		=&\mc F(L_{\gamma\lambda}f)(\sigma+\lambda^*)|\Pf(\sigma+\lambda^*)|^\frac{1}{2}=\pi_{\sigma+\lambda^*}(\gamma\lambda)\mc Ff(\sigma+\lambda^*)|\Pf(\sigma+\lambda^*)|^\frac{1}{2}\\
		=&e^{2\pi  i \langle \sigma,\lambda\rangle}\pi_{\sigma+\lambda^*}(\gamma)\mc Ff(\sigma+\lambda^*)|\Pf(\sigma+\lambda^*)|^\frac{1}{2}\\
		=&e^{2\pi  i \langle \sigma,\lambda\rangle}\left(\tilde\pi_{\sigma}(\gamma)\mc T^{\Lambda}f(\sigma)\right)(\lambda^*), \ 
	\end{align*}
since $\mc F (L_{\gamma} f )(\sigma)=\pi_{\sigma}(\gamma) \mc F f (\sigma)$.
\end{proof}
\begin{rk}
	For the case of Heisenberg group $\mathbb H^d$,  the unitary map $\mc T^{\Lambda}$ takes the form:
	$$\mc T^{\Lambda}: L^2(\mathbb H^d)\ra L^2([0,1) ,\ell^2(\mathbb Z, L^2(\mathbb R^d))),     \quad  \mc T^{\Lambda}f(\xi)(n)=\mc Ff(\xi+n)|\xi+n|^{d/2}, \ a.e. \ \xi\in [0,1), n\in \mathbb Z, f \in L^2 (\mathbb H^d).$$
	and the intertwining property of $\mc T^{\Lambda}$ gives $
	\mc T^{\Lambda}(L_{n \gamma}f)(\xi)=e^{2\pi i n \xi} \,  \tilde{\pi}_{\xi}(\gamma)\mc T^{\Lambda}f(\xi)$, where  $\gamma  \in A \mathbb Z^d \times B \mathbb Z^d$ such that  $AB^t\in \mathbb Z$.
	\end{rk}
Next we discuss the range function $J$ for a $\Gamma\Lambda$-invariant space.  The \textit{range function} is a mapping $J: \Sigma\ra \ \left\{\mbox{closed subspaces of }{\ell^2( \Lambda^\perp,\mc HS(L^2(\mathbb R^d)))}\right\}.$ It is  measurable if the projection map $P(\sigma):L^2(G)\ra J(\sigma)$ is weakly measurable i.e., for each $a,b\in \ell^2( \Lambda^\perp,\mc HS(L^2(\mathbb R^d))$
and $\sigma\mapsto\langle P(\sigma)a, b\rangle$ is measurable.
The space $W$ can be expressed as follows: $
W= \{\varphi\in L^2(G):\mc T^{\Lambda}\varphi(\sigma)\in J(\sigma), \ \mbox{for a.e.} \ \sigma \in \Sigma\} \ and \ \tilde \pi_{\sigma}(\Gamma)\subset J(\sigma).$ Also, there is bijection $W\mapsto J $. We refer  \cite{bar2014bracket, bownik2015structure,christensen2016introduction, currey2014characterization,ron1995frames} for more details about  shift invariant spaces and associated  range functions for the abelian and non-abelian  setups.

\begin{prop}\label{range  function} 
	The  range function  $J$ associated to  the $\Gamma\Lambda$-invariant space $W=\mc S^{\Gamma\Lambda}(\msc A)$ satisfies
	\begin{align}\label{J}
J(\sigma)=\ol{\sp} \{\mc T^{\Lambda}(L_{\gamma}\varphi)(\sigma):\varphi\in \msc A, \gamma\in\Gamma\},  \ \mbox{a.e.} \ \sigma\in \Sigma.
	\end{align}
\end{prop}

	\begin{proof}	From the intertwining property (\ref{eq:intertwining}), we get $\mc T^{\Lambda} (L_{\gamma\lambda}\varphi)(\sigma)= e^{2\pi i\langle\sigma,\lambda \rangle} \tilde\pi(\gamma)\mc T^{\Lambda}\varphi(\sigma)$,    for  $\gamma\lambda\in \Gamma\Lambda,$ and a.e. $\sigma\in \Sigma$,   and hence,    $\mc T^{\Lambda} (\mc S^{\Gamma\Lambda}(\msc A))$ is invariant under exponential and $\tilde \pi(\Gamma) \mc T^{\Lambda} (\mc S^{\Gamma\Lambda}(\msc A))\subset \mc T^{\Lambda} (\mc S^{\Gamma\Lambda}(\msc A))$. Therefore, we get the result by observing $\mc T^{\Lambda} (\mc S^{\Gamma\Lambda}(\msc A))=M_{J}$, where the space $M_J$ is defined  by  
		\begin{align}\label{MJ}
		M_J=\{f\in L^2( \Sigma, \ell^2(  \Lambda^\perp,\mc HS(L^2(\mathbb R^d)): f(\sigma)\in J(\sigma) \ \mbox{for a.e.}\ \sigma\in \Sigma\},  
		\end{align}
	  for the range function $J$ given in (\ref{J}).  For this let us consider a function   $g\in \mc T^{\Lambda} (\mc S^{\Gamma\Lambda}(\msc A))$. Choose a sequence $(g_i)$   converging  to $g$ such that 
		${\mc T^{\Lambda}}^{-1}g_i\in \sp\{L_{\gamma\lambda}\varphi:\gamma\lambda\in \Gamma\Lambda, \varphi\in \msc A\}$. Then   we have 
		$g_i(\sigma)\in J(\sigma)$ in view of  (\ref{eq:intertwining}), and hence $g(\sigma)\in J(\sigma)$ since $J(\sigma)$ is closed. Therefore,  $g\in M_{J}$, i.e.,  $\mc T^{\Lambda} (\mc S^{\Gamma\Lambda}(\msc A))\subset M_{J}$. For the equality $\mc T^{\Lambda} (\mc S^{\Gamma\Lambda}(\msc A))= M_{J}$, we need to show $\mc T^{\Lambda} (\mc S^{\Gamma\Lambda}(\msc A))^\perp \cap  M_{J}=0.$ Choose a function $h\in \mc T^{\Lambda} (S^{\Gamma\Lambda}(\msc A))^\perp \cap  M_{J}$.  
		 Then  for any $f\in  \mbox{span}\{\mc T^{\Lambda} (L_{\gamma}\varphi):\gamma\in\Gamma, \varphi\in\msc A\}$ and $\lambda\in \Lambda$, we have $e^{2\pi i\langle \cdot, \lambda\rangle } f(\cdot)\in \mc T^{\Lambda} (\mc S^{\Gamma\Lambda}(\msc A)) $, and then  we obtain
			\begin{align*}
		0=\int_{\Sigma}\langle e^{2\pi i\langle \cdot, \lambda\rangle }f(\sigma),h(\sigma) \rangle d\sigma=\int_{\Sigma} e^{2\pi i\langle \cdot, \lambda\rangle }\langle f(\sigma),h(\sigma) \rangle d\sigma,  
		\end{align*}
Hence all the Fourier coefficients of a scalar function given by $\sigma\mapsto\langle f(\sigma),h(\sigma) \rangle $ are zero.
	Thus,    $\langle f(\sigma),h(\sigma) \rangle=0$ a.e.  $\sigma\in \Sigma$ and  $f(\sigma)\in  J(\sigma)$, i.e.,  $h(\sigma)\in  J(\sigma)^\perp$, a.e. $\sigma\in \Sigma$.  
	\end{proof}

Employing the Proposition \ref{range  function} we characterize a member of  $\mc S^{\Gamma\Lambda}(\varphi)$ with the help of Plancherel transform.

\begin{prop}\label{decomposition}  For 
 $f\in \mc S^{\Gamma\Lambda}(\varphi)$,   the  Plancherel transform   of  $f$ is given  by 
\begin{align}\label{PTf}
\mc  F f(\omega)=\sum_{\gamma\in \Gamma} \beta_{\gamma}(\omega)\mc F ({L_{\gamma}\varphi})(\omega), \quad a.e. \ \omega\in \mathfrak z^*, 
\end{align}
where $\beta_\gamma$ is a $\Lambda^\perp$--periodic function. Conversely if $\beta_{\gamma}$ is an $\Lambda^\perp$-periodic function such that 
$$\sum_{\gamma\in \Gamma} \beta_{\gamma}(\cdot)\mc F({L_{\gamma}\varphi})(\cdot)\in L^2(\mathfrak z^*;{\mc HS}({ L^2(\mathbb R^d)}),
$$
 then the function $f$ defined by (\ref{PTf})  is a members of  $\mc  S^{\Gamma\Lambda}(\varphi).$
\end{prop}
\begin{proof}  Applying the Plancherel transform followed by periodization $\mc T^{\Lambda}$ on a function $f \in \mc  S^{\Gamma\Lambda}(\varphi)$, we get 
\begin{align}\label{Tf}
\mc T^{\Lambda}f(\sigma)=\mc T^{\Lambda}(Pf)(\sigma)=P(\sigma)\mc T^{\Lambda}f(\sigma)=\sum_{\gamma\in \Gamma} \frac{\langle\mc T^{\Lambda}f(\sigma),\mc T^{\Lambda}(L_{\gamma}\varphi)(\sigma)\rangle }{\|\mc T^{\Lambda}(L_{\gamma}\varphi)(\sigma)\|^2}\mc T^{\Lambda}(L_{\gamma}\varphi)(\sigma), \  a.e. \ \sigma \in \Sigma, 
\end{align}
 in view of Proposition \ref{range  function}  and commutativity of $P$ and $\mc T^{\Lambda}$, where $P$  and $P(\sigma)$  are  orthogonal projections on $\mc S^{\Gamma\Lambda}(\varphi)$ and  $J(\sigma)$,  respectively. The above expression (\ref{Tf})  can be written as $ 
\mc T^{\Lambda}f(\sigma)=\sum_{\gamma\in \Gamma} \beta_{\gamma} (\sigma) \mc T^{\Lambda}L_{\gamma}\varphi(\sigma),$ for a.e. $\sigma \in \Sigma$, 
  where   the $\Lambda^\perp$--periodic function $\beta_\gamma$ is defined by 
  $$\beta_{\gamma}(\sigma)=\begin{cases}
\frac{\langle\mc T^{\Lambda}f(\sigma),\mc  T^\Lambda (L_{\gamma}\varphi)(\sigma)\rangle }{\|\mc  T^\Lambda (L_{\gamma}\varphi)(\sigma)\|^2}=\mathlarger\sum_{\lambda^* \in \Lambda^\perp}\frac{\langle\mc T^{\Lambda}f(\sigma)(\lambda^*),\mc  T^\Lambda (L_{\gamma}\varphi)(\sigma)(\lambda^*)\rangle }{\|\mc  T^\Lambda (L_{\gamma}\varphi)(\sigma)\|^2}, \ &\sigma \in\Sigma_{\varphi}(\sigma)=\left\{\sigma\in\Sigma:\|\mc  T^\Lambda (L_{\gamma}\varphi)(\sigma)\|^2\neq 0\right\}\\
0, \ &\mbox{otherwise.}
\end{cases}
$$
 The function $\beta_{\gamma}$  can  be extended  periodically on  $\widehat Z$ since  $\Lambda^\perp$ is a tiling partner of $\Sigma$ for $\widehat Z$.  Also observe that for any $w\in \widehat Z$, there exists unique $\sigma\in \Sigma, \lambda^*\in \Lambda^\perp$ such that $w=\sigma+\lambda^*$, and hence   from (\ref{Tf}), we obtain
 \begin{align*}
 \mc F f(w)|\Pf(\omega)| &=\mc F f(\sigma+\lambda^*)|\Pf(\sigma+\lambda^*)|=\mc T^{\Lambda}f(\sigma)(\lambda^*)=\sum_{\gamma\in \Gamma} \beta_{\gamma}(\sigma) \mc T^{\Lambda}(L_{\gamma}\varphi)(\sigma)(\lambda^*)\\
 &=\sum_{\gamma\in \Gamma} \beta_{\gamma}(\sigma+\lambda^*) \mc F ({L_{\gamma}\varphi})(\sigma)(\lambda^*)|\Pf(\sigma+\lambda^*)|\\
 &=\sum_{\gamma\in \Gamma} \beta_{\gamma}(w) \mc F ({L_{\gamma}\varphi})(w)|\Pf(\omega)|.
\end{align*} 
 The converse part follows  from  the above calculations by writing   
$\mc F f(\cdot)=\sum_{\gamma\in \Gamma} \beta_{\gamma}(\cdot)\mc F ({{L_{\gamma}\varphi}})(\cdot)$, in the form 
$\mc  T^\Lambda f(\sigma)=\sum_{\gamma\in \Gamma} \beta_{\gamma}(\sigma)\mc  T^\Lambda (L_{\gamma}\varphi)(\sigma)$, and noting   $\mc  T^\Lambda f(\sigma)\in J(\sigma)$, gives  $f\in \mc S^{\Gamma\Lambda}(\varphi)$  from   Proposition \ref{range function}.
\end{proof}

\begin{proof}[Proof of Lemma \ref{lemmaWj}] For  $j\in \mc J$, let us consider the space  $V_j^\Theta =\{f\in L^2(G): \widehat f= \chi_{\mc H_j^{\Theta}}\widehat g, \ \mbox{for some} \  g \in W\}$, 
where $\mc H_j^{\Theta}=\Sigma+j+\Theta^\perp$. To prove it as a $\Gamma\Theta$-invariant subspace of $L^2(G)$, we first  assume a sequence $(\varphi_k)$ in $V_j^\Theta$ converging to $\varphi \in L^2(G).$ Then $\varphi\in W$ since $V_j^\Theta\subset W$ and $W$ is closed, and hence,  $\varphi\in V_j^\Theta$. This follows by writing $\widehat \varphi=\chi_{\mc H_{j}^\Theta}\widehat \varphi$ since  $\|\varphi_k-\varphi\| \rightarrow 0$ implies  $\widehat \varphi\chi_{(\mc H_{j}^\Theta)^c}=0$    from 
\begin{equation*}
\|\varphi_k- \varphi\|^2  \geq  \|(\widehat \varphi_k-\widehat \varphi)\chi_{(\mc H_{j}^\Theta)^c}\|^2 \geq \|\widehat \varphi\chi_{(\mc H_{j}^\Theta)^c}\|^2, \quad \mbox{where c denotes complement of the set}.
\end{equation*}
 Therefore,  $V_j^\Theta$ is closed. Further we observe that if $f\in {V_j^\Theta}$,  we have $\widehat f=\chi_{\mc H_{j}^\Theta} \widehat g$,  for some $g\in W$, and hence for $\theta \in \Theta$ and $\gamma \in \Gamma$, we can write $ e^{2\pi i\langle 
 	\omega,\theta \rangle} \mc F(L_{\gamma}f)(\omega)=\chi_{\mc H_{j}^\Theta}(\omega)e^{2\pi i\langle \omega,
 	\theta \rangle}\mc F(L_{\gamma}g)(\omega),$ for $\omega \in \mathfrak z^*$ since  $\mc F(L_{\gamma}g)(\omega)= \pi_{\omega} (\gamma) \widehat g(\omega)$. For the $\Gamma \Theta$-invariant it suffices to show $e^{2\pi i\langle \omega,
 	\theta\rangle}\mc F(L_{\gamma}g)(\omega) \in \mc F(W)$ that  gives   $e^{2\pi i\langle \omega,
 	\theta\rangle} \mc F(L_{\gamma}f)(\omega) \in \mc F (V_j^\Theta)$.  Observe  that $e^{2\pi i\langle \omega,
 	\theta\rangle}\mc F(L_{\gamma}g)(\omega) \in \mc F(\mc S^{\Gamma\Lambda} (g)) \subset  \mc F(W)$ 
 due to the converse part of   Proposition \ref{decomposition},  provided  $t_\theta (\omega):=e^{2\pi i\langle \omega,
 	\theta\rangle}  \ a.e. \ \omega \in \mc H_{j}^\Theta$, is a $\Lambda^\perp$-periodic function.  
 Since $e^{2\pi i\langle \cdot,
 	\theta\rangle}$ is $\Theta^\perp$-periodic, we have 
 $e^{2\pi i\langle 
 	\sigma +j,\theta \rangle}=e^{2\pi i\langle 
 	\sigma +j+\theta^*,\theta \rangle}$, for a.e $\sigma \in \Sigma$,  $j \in \mc J$,   and for every  $\theta^* \in \Theta^\perp$, and then for each $\lambda^* \in \Lambda^\perp$, we define 
$t_\theta(\sigma +\lambda^*)=e^{2\pi i\langle 
	\sigma +j,\theta\rangle}$, a.e. $\sigma \in \Sigma$. Thus the function $t_\theta$ is  $\Lambda^\perp$-periodic  on $\Sigma$,   can be extended to $\mathfrak z*$ since $\Lambda^\perp$ is tiling partner of $\Sigma$ for $\mathfrak z^* \cong  \widehat Z$. 
\end{proof}

\begin{proof}[{Proof of Theorem \ref{extra-invariance result}}] For each $j\in \mc J$, if $V_j^\Theta\subset W$, the space $V_j^\Theta$ is $\Gamma \Theta$-invariant from Lemma \ref{lemmaWj}, and hence the space  $\bigoplus_{j\in \mc J} V_j^\Theta \subset W$ is so. Since  $\{\mc H_{j}^\Theta\}_{j\in \mc J}$ is a tiling  of $\widehat Z\cong \mathfrak z^*$, therefore any element $f \in W$ can be written as $\widehat f(\omega)=\sum_{j\in \mc J} \widehat g_j(\omega)$, a.e. $\omega \in \mathfrak z^*$,  where $\widehat g_j=\widehat f \chi_{\mc H_{j}^\Theta}.$  By the definition of $V_j^\Theta$, $g_j \in V_j^\Theta$  for every $j \in \mc J$ and hence $f \in \bigoplus_{j \in \mc J} V_j^\Theta$. Therefore, $W$ is  $\Gamma\Theta$-invariant.

	\noindent 	Conversely,  let us assume that the $\Gamma
	\Lambda$-invariant space $W$ is   $\Gamma
	\Theta$-invariant. For  $V_j^\Theta\subset W$, we choose $f\in V_j^\Theta$. Then,  we have    $\widehat f=\chi_{\mc H_{j}^\Theta}\widehat g$, for some $g\in W$. Employing  the Plancahrel transform followed by periodization   $\mc T^{\Theta}$ from   $L^2(G)$ to $L^2( \mc D, \ell^2(  \Theta^\perp,\mc HS(L^2(\mathbb R^d)))$ given by 	$\mc T^{\Theta}f(\delta)(\theta^*)=\mc Ff(\delta+\theta^*)|\Pf(\delta+\theta^*)|^{1/2}, \   f\in L^2(G),  \theta^* \in \Theta^\perp \  \mbox{and a.e.} \ \delta \in \mc D$, where $\mc D$ is the  Borel section  of $\widehat Z/\Theta^\perp$, 
	we obtain 
	$$\mc T^{\Theta} f(\delta)(\theta^*)=\chi_{\mc H_{j}^\Theta}(\delta)\mc T^\Theta g(\delta)(\theta^*),$$ 
	since $\chi_{\mc H_{j}^\Theta}$ is $\Theta^\perp$-periodic due to the definition of $\mc H_{j}^\Theta$ in (\ref{Hj}). Then  for a.e. $\delta\in\mc D$, we have $\mc T^\Theta (L_{e} g)(\delta) \in  J^\Theta(\delta)$, where    $J^\Theta(\delta)=\ol{\sp} \{\mc T^{\Theta}(L_{\gamma}g)(\delta): \gamma\in\Gamma\},$ and hence  $\mc T^\Theta f(\delta)\in J^\Theta(\delta)$, for a.e. $\delta\in\mc D$. Thus $f\in \mc S^{\Gamma\Theta}(g)\subset W$ due to Proposition \ref{range  function}. \end{proof}

	\begin{proof}[Proof of Theorem\ref{TI_extra-invariance result}]
For  $j\in \mc J$, assume  $V_j^\Theta =\{f\in L^2(G): \widehat f= \chi_{\mc H_j^{\Theta}}\widehat g, \ \mbox{for some} \  g \in W\}$, and $W_j=\{f\in L^2(G): \mbox{supp}(\widehat  f)\subset {\mc H^\Theta_j}\}$, where $\mc H_j^{\Theta}=\Sigma+j+\Theta^\perp$. Let $P_{j}$ be the orthogonal projection on $W_j$. Then 
	$$P_{j}(\mc S^{\Gamma\Lambda}(\msc A))=\{f^j:\widehat f^j=\widehat f\chi_{\mc H^\Theta_j}, f\in \mc S^{\Gamma\Lambda}(\msc A) \}=V_j^\Theta,$$
	whose associated range function  is $J_{V_j^\Theta}(\sigma)=\overline{\sp}\{\mc T^{\Lambda}(L_{\gamma}\varphi^j)(\sigma):\varphi\in\msc A, \gamma\in \Gamma, \widehat \varphi^j=\widehat \varphi \chi_{\mc H^{\Theta}_j}\} $,  for a.e. $\sigma \in \Sigma$. Therefore from Theorem \ref{extra-invariance result},  $S^{\Gamma\Lambda}(\msc A)$ is a $\Gamma \Theta$-invariant if and only if $V_j^\Theta \subset \mc S^{\Gamma\Lambda}(\msc A)$, for each $j \in \mc J$.    Further it is   equivalent to  $J_{V_j^\Theta}(\sigma) \subset J(\sigma)$,  for a.e. $\sigma\in \Sigma, \ \mbox{for all} \ j\in \mc J$, where $J(\sigma)$ is the range function associated to $\mc S^{\Gamma\Lambda}(\msc A)$, see Proposition \ref{range  function} and \cite{bownik2015structure}. Thus the result follows.  
\end{proof}

 \begin{proof}[Proof of Theorem \ref{dimension}] From Theorem \ref{extra-invariance result},  we have $ W= \displaystyle\bigoplus_{j\in \mc J}V_j^\Theta$ if $W$ is $\Gamma\Theta$-invariant. Then   for a.e. $\sigma\in \Sigma$,  the range function satisfies $J_W(\sigma)=\bigoplus_{j\in \mc J} J_{V_j^\Theta}(\sigma)$, follows by observing the orthogonality of $J_{V_j^\Theta}(\sigma)$ and $J_{V_{j'}^\Theta}(\sigma)$, for $j\neq j'$,
 	since $\{\mc H_{j}^\Theta\}_{j\in \mc J}$ is a tiling  of $\widehat Z\cong \mathfrak z^*$. Hence  $\dim_W(\sigma)=\sum_{j\in\mc J} \dim_{V_j^\Theta}(\sigma)$, a.e. $\sigma\in \Sigma$. 
 	
 	For the converse part, first observe that the $\Gamma\Lambda$-invariant space $W$ is contained in $\bigoplus_{j \in \mc J} V_j^\Theta$. This follows by writing  $f \in W$ as $\widehat f(\omega)=\sum_{j\in \mc J} \widehat g_j(\omega)$, a.e. $\omega \in \mathfrak z^*$,  where $\widehat g_j=\widehat f \chi_{\mc H_{j}^\Theta}$ since $\{\mc H_{j}^\Theta\}_{j\in \mc J}$ is a tiling  of $\widehat Z\cong \mathfrak z^*$. Then the range function satisfies $J_W(\sigma)\subset \bigoplus_{j\in \mc J}J_{V_j^\Theta}(\sigma)$, and hence we have $J_W(\sigma)=\bigoplus_{j\in \mc J}J_{V_j^\Theta}(\sigma)$ for a.e. $\sigma\in \Sigma$ due to the condition $\dim_W(\sigma)=\sum_{j\in\mc J} \dim_{V_j^\Theta}(\sigma)$, a.e. $\sigma\in \Sigma$. Therefore we get $J_{V_j^\Theta}(\sigma)\subset J_W(\sigma)$, for each $j \in \mc J$, i.e, $V_j^\Theta \subset W$, for all $j$. Thus $W$ is $\Gamma\Theta$-invariant follows from Theorem \ref{extra-invariance result}.
\end{proof}

\begin{proof}[Proof of Theorem \ref{measure}]
		For  $\theta^*\in \Theta^\perp$ and $\varphi \in \msc A$, we first estimate the measure of following expression:  
	\begin{align*}
	\mu\left(\left\{(\sigma, j) \in \Sigma \times \mc J:\widehat{\varphi}(\sigma+j+\theta^*)\neq 0\right\}\right)&=\mu\left(\left\{(\sigma, j) \in \Sigma \times \mc J: \tilde{\pi}_{\sigma+j+\theta^*}(\gamma)\widehat{\varphi}(\sigma+j+\theta^*)\neq 0\right\}\right), \mbox{for any} \ \gamma \in \Gamma\\
	&=\mu\left(\left\{(\sigma, j) \in \Sigma \times \mc J: \mc F (L_{\gamma}\varphi)(\sigma+j+\theta^*)\neq 0\right\}\right)\\
	  &=\int_{\Sigma} \left|S_\sigma^{\mc J} \right| d\sigma,
	\end{align*}
where the set $S_\sigma^{\mc J} :=\{j \in \mc J :\mc F (L_{\gamma}\varphi)(\sigma+j+\theta^*)\neq 0\}$ and $|S_\sigma^{\mc J}|$ denotes the cardinality of $S_\sigma^{\mc J}$.  For a.e. $\sigma \in \Sigma$, the set $S_\sigma^{\mc J}$ is contained in the set $\{j\in \mc J:\dim_{ V_j^\Theta}(\sigma)\neq 0\}$  since $\dim_{ V_j^\Theta}(\sigma)=\dim J_{V_j^\Theta}(\sigma)$, where  
 $J_{V_j^\Theta}(\sigma)=\overline{\sp}\{\mc T^{\Lambda}(L_{\gamma}\varphi^j)(\sigma):\varphi\in\msc A, \gamma\in \Gamma, \widehat \varphi^j=\widehat \varphi \chi_{\mc H^{\Theta}_j}\}$. Then, we have  
 $$
 |S_\sigma^{\mc J}| \leq |\{j\in \mc J:\dim_{ V_j^\Theta}(\sigma)\neq 0\}|  \leq \sum_{j\in \mc J} \dim_{V_j^\Theta}(\sigma)=\dim_W(\sigma), \ a.e. \ \sigma \in \Sigma.
 $$
Since the set   $\{\Sigma+j+\theta^*\}_{j\in \mc J, \theta^*\in \Theta^\perp}$
is a tiling set  for $\widehat  Z$, therefore for a fixed $\sigma\in\Sigma$ and $j\in \mc J$   there is a  unique
 $\theta_{\sigma,j}^*\in \Theta^\perp$ such that 
$\sigma+j+\theta_{\sigma,j}^*\in \mc D$, and hence we have 
{\small\begin{align*}
	\mu\left(\left\{ \delta \in \mc D:\widehat \varphi (\delta)\neq 0\right\}\right) 
	& =\sum_{j\in \mc J} \mu\left(\left\{ \sigma\in\Sigma:\widehat{\varphi}(\sigma+j+\theta_{\sigma,j}^*)\neq 0\right\}\right)\\
&=	\mu\left(\left\{(\sigma, j) \in \Sigma \times \mc J:\widehat{\varphi}(\sigma+j+\theta^*_{\sigma,j})\neq 0\right\}\right)\\
&=\int_{\Sigma} \left|S_\sigma^{\mc J} \right| d\sigma\\
&\leq \int_{\Sigma}  \sum_{j\in \mc J} \dim_{V_j^\Theta}(\sigma) d\sigma= \int_{\Sigma} \dim_W(\sigma) d\sigma\\ 
&=\sum_{m=0}^{nk} m \, \mu(\Sigma_{m})   \leq nk, 
\end{align*}}
where  $|\Gamma|=k$ and   $\Sigma_{m}=\{\sigma\in\Sigma: \dim _W(\sigma)=m \}$.  Thus the result follows.
\end{proof}
We have a immediate consequence for the singly generated system.
\begin{cor}\label{support} Let $\Sigma$ and $\mc J$ be the Borel sections of $\widehat Z/\Lambda^\perp$ and  $\Lambda^\perp/\Theta^\perp$, respectively,  such that the cardinality $|\mc J|$ of $\mc J$ and measure $\mu(\Sigma)$ of $\Sigma$ satisfies the relation
	 $$
	   (|\mc J|\mu(\Sigma) -k)>0, \ \mbox{where} \ k \ \mbox{is the cardinality of a nonempty set}\  \Gamma.
	 $$
When  the space $\mc S^{\Gamma\Lambda}(\varphi)$  becomes  $\Gamma\Theta$-invariant,   then the Plancherel transform $\widehat \varphi$  of $\varphi$ satisfies the following relation:  
	$$\mu(\{\omega\in \mathfrak z^* :\widehat \varphi(\omega)=0\})\geq  |\Theta^\perp|(|\mc J|\mu(\Sigma) -k)>0.$$
\end{cor}
\begin{proof} Considering a Borel section   $\mc D$  of $\widehat Z/\Theta^\perp$ and noting $\widehat Z \cong \mathfrak z^*$, we have the following from Theorem \ref{measure}:
	\begin{align*}
	\mu(\{\omega\in \mathfrak z^* :\widehat \varphi(\omega)=0\}) &=\sum_{\theta^*\in \Theta^\perp} \mu(\{\delta \in \mc D+\theta^*:\widehat \varphi(\delta)=0\})=\sum_{\theta^*\in \Theta^\perp} \mu[(\mc D+\theta^*) \backslash (\{\delta \in \mc D+\theta^*:\widehat \varphi(\delta)\neq 0\})]\\
	&=\sum_{\theta^*\in \Theta^\perp}  \mu(\mc D) - \sum_{\theta^*\in \Theta^\perp}  \mu(\{\delta\in \mc D:\widehat \varphi(\delta)\neq 0\})\\
	&=\sum_{\theta^*\in \Theta^\perp}  \sum_{j\in \mc J}\mu(\Sigma+j) - \sum_{\theta^*\in \Theta^\perp}\mu(\{y\in \mc D:\widehat \varphi(y)\neq 0\})\\
		& \geq|\Theta^\perp||\mc J | \mu(\Sigma)-k |\Theta^\perp|=|\Theta^\perp|[ |\mc J|\mu(\Sigma) -k]>0.
	\end{align*}
Thus the result follows.
\end{proof}
 
\begin{rk}
 For the Heisenberg group $\mathbb  H^d$, the center $Z=\mathbb R$,  the uniform lattice $\Lambda=\mathbb Z$ and  the extra-invariance set $\Theta=\frac{1}{N} \mathbb Z$. Then the annihilators of $\Lambda$ and $\Theta$ are $\Lambda^\perp=\mathbb Z$ and  $\Theta^\perp=N\mathbb Z$, respectively.  Consider $\Sigma=[0, 1)$ and $\mc J=\{0, 1, 2, \cdots, N-1\}$ be the Borel sections of $\widehat Z/ \Lambda^\perp=\mathbb R/\mathbb Z$  and  $ \Lambda^\perp/ \Theta^\perp= \mathbb Z/ N\mathbb Z$, and choose   $\Gamma=\{0\}$. Then  $\mc D=[0, N)$  is the Borel section of $\widehat Z/\Theta^\perp$ and the estimate (\ref{estimate}) mentioned in Theorem \ref{measure} can be written as $ 	\mu(\{\xi\in [0, N): \widehat \varphi_i(\xi)\neq 0\})  \leq n$,  for all  $i \in \{1, 2, \cdots, n\}, $ since the cardinality of $\Gamma$ is $k=1$. For $N>1$, when the space $\mc S^{\Gamma\Lambda}(\varphi)$ becomes $\Gamma\Theta$-invariant,   the measure of  the set $\{\xi\in \mathbb R :\widehat \varphi(\xi)=0\}$ is infinite from Corollary \ref{support}.
\end{rk}

\providecommand{\bysame}{\leavevmode\hbox to3em{\hrulefill}\thinspace}
\providecommand{\MR}{\relax\ifhmode\unskip\space\fi MR }
\providecommand{\MRhref}[2]{%
	\href{http://www.ams.org/mathscinet-getitem?mr=#1}{#2}
}
\providecommand{\href}[2]{#2}


\end{document}